\documentclass[12pt]{article}
\usepackage{amsfonts}
\usepackage{amsmath,amssymb,amsthm}
\usepackage[usenames]{color}
\usepackage{mathrsfs}
\usepackage{syntonly}
\usepackage[top=1in,bottom=1.4in,left=1in,right=1in]{geometry}
\usepackage[all,cmtip]{xy}
\usepackage{bbm}
\usepackage{mathrsfs}
\usepackage{dsfont}
\usepackage{enumitem}
\usepackage{fouridx}
\usepackage[stable]{footmisc}
\usepackage{array}
\usepackage{caption}
\usepackage{latexsym}
\usepackage{extarrows}
\usepackage{amsfonts}
\usepackage{amsmath}
\usepackage{amsthm}
\usepackage{amssymb}
\usepackage{latexsym}
\usepackage{indentfirst}
\allowdisplaybreaks

\newcommand{\C}{\mathbb{C}}
\newcommand{\Z}{\mathbb{Z}}
\newcommand{\N}{\mathbb{N}}

\newcommand{\g}{\mathfrak{g}}
\newcommand{\W}{\mathcal{W}}

\newcommand{\yg}{Y(\mathfrak{g})}
\newcommand{\ysl}{Y(\mathfrak{sl}_2)}
\newcommand{\nysl}{\mathfrak{sl}_2}
\newcommand{\nyn}{\mathfrak{sl}_{l+1}}
\newcommand{\nyo}{\mathfrak{so}(2l,\C)}

\newcommand{\nysp}{\mathfrak{sp}(2l,\C)}

\newcommand{\nyso}{\mathfrak{so}(2l+1,\C)}
\newcommand{\lap}{\lambda_{\pi}}

\numberwithin{equation}{section}
\newtheorem{theorem}{Theorem}[section] %to number theorems, etc. within  sections
\newtheorem{proposition}[theorem]{Proposition}
\newtheorem{corollary}[theorem]{Corollary}
\newtheorem{lemma}[theorem]{Lemma}
\newtheorem{definition}[theorem]{Definition}
\newtheorem{remark}[theorem]{Remark}

\newcommand{\DynkinEs}
{
\xy
(0.5,0)*{};(7.5,0)*{} **\dir{-};
(8.5,0)*{};(15.5,0)*{} **\dir{-};
(16.5,0)*{};(23.5,0)*{} **\dir{-};
(24.5,0)*{};(31.5,0)*{} **\dir{-};
(16,0.5)*{};(16,5.5)*{} **\dir{-};
(0,0)*{\circ}; (8,0)*{\circ}; (16,0)*{\circ};
(24,0)*{\circ}; (32,0)*{\circ};
(16,6)*{\circ};
(0,-3)*{1}; (8,-3)*{3}; (16,-3)*{4};
(24,-3)*{5}; (32,-3)*{6};
(14,5)*{2};
\endxy
}

\newcommand{\DynkinEsv}
{
\xy
(0.5,0)*{};(7.5,0)*{} **\dir{-};
(8.5,0)*{};(15.5,0)*{} **\dir{-};
(16.5,0)*{};(23.5,0)*{} **\dir{-};
(24.5,0)*{};(31.5,0)*{} **\dir{-};
(32.5,0)*{};(39.5,0)*{} **\dir{-};
(16,0.5)*{};(16,5.5)*{} **\dir{-};
(0,0)*{\circ}; (8,0)*{\circ}; (16,0)*{\circ};
(24,0)*{\circ}; (32,0)*{\circ};
(40,0)*{\circ};
(16,6)*{\circ};
(0,-3)*{1}; (8,-3)*{3}; (16,-3)*{4};
(24,-3)*{5}; (32,-3)*{6}; (40,-3)*{7};
(14,5)*{2};
\endxy
}

\newcommand{\DynkinEe}
{
\xy
(0.5,0)*{};(7.5,0)*{} **\dir{-};
(8.5,0)*{};(15.5,0)*{} **\dir{-};
(16.5,0)*{};(23.5,0)*{} **\dir{-};
(24.5,0)*{};(31.5,0)*{} **\dir{-};
(32.5,0)*{};(39.5,0)*{} **\dir{-};
(40.5,0)*{};(47.5,0)*{} **\dir{-};
(16,0.5)*{};(16,5.5)*{} **\dir{-};
(0,0)*{\circ}; (8,0)*{\circ}; (16,0)*{\circ};
(24,0)*{\circ}; (32,0)*{\circ};
(40,0)*{\circ}; (48,0)*{\circ};
(16,6)*{\circ};
(0,-3)*{1}; (8,-3)*{3}; (16,-3)*{4};
(24,-3)*{5}; (32,-3)*{6}; (40,-3)*{7};
(48,-3)*{8}; (14,5)*{2};
\endxy
}

\newcommand{\DynkinF}
{
\xy
(0.5,0)*{};(7.5,0)*{} **\dir{-};
{\ar@{=>} (8.5,0)*{}; (15.5,0)*{}};
(16.5,0)*{};(23.5,0)*{} **\dir{-};
(0,0)*{\circ}; (8,0)*{\circ};
(16,0)*{\circ}; (24,0)*{\circ};
(0,-3)*{1}; (8,-3)*{2};
(16,-3)*{3};(24,-3)*{4};
\endxy
}

\newcommand{\DynkinG}
{
\xy
(0.5,0.5)*{};(7,0.5)*{} **\dir{-};
(0.5,0)*{};(7.5,0)*{} **\dir{-};
(0.5,-0.5)*{};(7,-0.5)*{} **\dir{-};
(6.5,1.5)*{};(7.5,0)*{} **\dir{-};
(6.5,-1.5)*{};(7.5,0)*{} **\dir{-};
(0,0)*{\circ}; (8,0)*{\circ};
(0,-3)*{1}; (8,-3)*{2};
\endxy
}
\begin{document}

%%%  Abbreviations

\title{\Large\bf  Braid group actions and tensor products for Yangians}
%\vspace{17 mm}
\author{{Yilan Tan}}

\date{} % 01 Jan 2015 - start.

\maketitle
\begin{abstract}
We introduce a braid group action on $l$-tuple of rational functions for the finite-dimensional representations of Yangians $\yg$, where $\g$ is a complex simple Lie algebra. It provides an efficient way to compute certain polynomials which allows us to provide a finite set of numbers at which the tensor product of Kirillov-Reshetikhin modules of Yangians may fail to be cyclic.
\end{abstract}

\vspace{10 mm}

%\tableofcontents
%%%\vspace{7 mm}
%%%

\vspace{39 mm}

%\noindent School of Mathematics and Statistics\newline
%Shandong University at Weihai, Weihai, Shandong, 264209, China\newline
%tanyanlan@gmail.com

%\tableofcontents

%\vspace{15 mm}

%\noindent
%A.M.\newline
%{\it School of Mathematics and Statistics\newline
%University of Sydney,
%NSW 2006, Australia\newline}
%{\tt alexm@maths.usyd.edu.au}

%\vspace{7 mm}

%\noindent
%E.R.\newline
%{\it LAPTH, Chemin de Bellevue, BP 110\newline
%F-74941 Annecy-le-Vieux cedex, France\newline}
%{\tt ragoucy@lapp.in2p3.fr}
\newpage

\section{Introduction}
Let $\g$ be a complex simple Lie algebra of rank $l$. For each $\g$, there exists a Hopf algebra $\yg$, called the Yangian of $\g$, which is a deformation of the universal enveloping algebra of the current algebra $\g[t]$. %The representation theory of $\yg$ has many applications in mathematics and physics. The most well known example is that from any finite-dimensional representation of a Yangian, one can construct a $R$-matrix which is a  solution of the quantum Yang-Baxter equation \cite{ChPr4}.
For each $l$-tuple of polynomials $\pi=\big(\pi_1(u),\ldots, \pi_{l}(u)\big)$, Drinfeld showed in \cite{Dr2} that there exists a unique finite-dimensional
irreducible representation $V(\pi)$ of the Yangian $\yg$. If there exists an $i\in I=\{1,\ldots, l\}$ such that
$\pi_i(u)=\big(u-a\big)\big(u-(a+1)\big)\ldots \big(u-(a+m-1)\big)$ and $\pi_j(u)=1$ for $j\neq i$, then $V(\pi)$ is called the Kirillov-Reshetikhin
module. In this case, $\pi$ is denoted by $\pi_{a,m}^{(i)}$. In particular, when $m=1$,
$V(\pi)$ is called an $i$-th fundamental representation, denoted by $V_{a}(\omega_i)$,  and the $l$-tuple of polynomials is denoted by $\pi_{i,a}$. %The structure of some Kirillov-Reshetikhin modules is known, see \cite{ChPr4, KiRe, Kl}.
When $\g=\nysl$, $V(\pi)$ can be realized as a tensor product of Kirillov-Reshetikhin modules \cite{ChPr3}. For $\g\neq \nysl$, the structure of $V(\pi)$ is less clear.
Chari and Pressley showed in \cite{ChPr4} that $V(\pi)$ is a sub-quotient of a tensor product of fundamental representations $V_{a_1}(\omega_{b_1})\otimes V_{a_2}(\omega_{b_2})\otimes \ldots\otimes V_{a_k}(\omega_{b_k})$,
where the tensor factors may take any order. In the context of quantum affine algebras $U_{q}(\hat{\g})$, the irreducibility of the tensor product of fundamental representations has been studied by many authors who used a variety of methods \cite{AkKa, FrMu, Ka, MoToZh, VaVa}. This question was generalized by Chari \cite{Ch3}, who gave a concrete irreducibility condition for a tensor product of Kirillov-Reshetikhin modules of quantum affine algebras. %These irreducibility criterions solved the
%structure problem for a wide class of finite-dimensional irreducible representations of quantum affine algebras.
We note that determining certain polynomials is the key to giving an irreducibility condition for the aforementioned tensor product and Chari connected it with defining a braid group action on the set of $l$-tuple of rational functions.

Adopting the methods in the Chari's paper, we provided an analogous concrete irreducibility condition
for the tensor product of fundamental representations of $\yg$ \cite{TaGu, Ta}, where $\g$ is either of classical type or of type $G_2$.
Since there is no analogous braid group action available
for Yangians at that time, we computed certain polynomials using some of the defining relations of $\yg$. In practice, it is tedious to compute certain polynomials in the same way when $\g$ is an exceptional
simple Lie algebra of type either $E$ or $F$, which is why we had not obtained an irreducibility condition in these cases at that moment. This motivated us to introduce a braid group action for Yangians.

We would like to mention here that it has been established by Varagnolo \cite{Va} that together with results of Nakajima \cite{Na}, $\yg$ and $U_{q}(\hat{\g})$ has the same finite-dimensional representation theory. The applications of quiver varieties is essential in these papers. More recently, Gautam and Valerio Toledano \cite{GaTo} provided an equivalence of meromorphic braided tensor categories, where the Yangians and quantum affine algebras endowed with the deformed Drinfeld coproducts. The aim of this paper is to give an independent proof of a cyclicity condition for a tensor product (the standard coproduct) of Kirillov-Reshetikhin modules of Yangians. %The treatment is analogue to the proofs in \cite{Ch3}. The main differences to the Chari's paper is to show that: certain polynomials are determined by this braid group action on the $l$-tuple of Drinfeld polynomials $\pi$.  The reason is that there is no analogue of the sixth relation in (2.4) in \cite{Ch3} for Yangians and the corresponding defining relation of $\yg$ is much difficult to apply in practice. We develop new techniques by using the formula in Corollary 1.5 in \cite{Le} (Lemma \ref{c7lecor15}) and some results from our previous papers \cite{TaGu, Ta}.

%Some experts may
%It is widely believed that q-character theory is a strong tool to study finite-dimensional representations
%of quantum affine algebras. Chari and Moura studied the q-characters for finite-dimensional representations of
%quantum affine algebras \cite{ChMo2} and provided closed formulas for the q-characters of the fundamental representations
%of the quantum affine algebras \cite{ChMo3}. One important technique used in their papers is the braid group action.
%In \cite{Kn}, Knight proposed a character theory for finite-dimensional representations of Yangians and suggested the study
%of the character of fundamental representations of $\yg$ is the first step. Since Knight's characters essentially agree
%with q-characters \big(P164, \cite{FrRe}\big), it is nature to ask if one can define
%a braid group action for Yangians. %What we knew for now are that the sixth relation in (2.4) in \cite{Ch3} plays an important rule to connect the braid group action and the certain polynomial, see Proposition 4.1 therein, and the analogous relation for $\yg$ is much more complicated.

This paper is organized as follows. In Section 2, the Yangian $\yg$ and its finite-dimensional representations are summerized. In Section 3, we define a
braid group action on the set of $l$-tuple of rational functions over complex numbers and show that the action is well defined. In Section 4, we show in Proposition \ref{ogpigf3p} that the eigenvalues of $h_{i}(u)$ on an extremal vector is governed by the braid group action defined in Section 3. Using this, one can find a cyclicity condition for the tensor product of finite-dimensional irreducible representations of $\yg$. In Section 5, we make this cyclicity condition explicit when the finite-dimensional irreducible representations are Kirillov-Reshetikhin modules. To conclude, we provide the structure of local Weyl modules of $\yg$, where $\g$ is simple Lie algebra of type either $E$ or $F$.

\section{Yangians and their representations}
In this section, we recall the definition of the Yangian $\yg$ and some results concerning its finite-dimensional representations.
\begin{definition}
Let $\g$ be a simple Lie algebra over $\C$ with rank $l$ and let $A=\left(a_{ij}\right)_{i,j\in I}$, where $I=\{1,2,\ldots, l\}$, be its Cartan matrix. Let $D=\operatorname{diag}\left(d_1,\ldots, d_{l}\right)$, $d_i\in \N$, such that $d_1, d_2,\ldots, d_l$ are co-prime and $DA$ is symmetric. The Yangian $Y\left(\frak g\right)$ is defined to be the
associative algebra with generators $x_{i,r}^{{}\pm{}}$,
$h_{i,r}$, $i\in I$, $r\in\Z_{\geq 0}$, and the following
defining relations:
\begin{equation*}\label{}
[h_{i,r},h_{j,s}]=0, \qquad [h_{i,0}, x_{j,s}^{\pm}]={}\pm
d_ia_{ij}x_{j,s}^{\pm}, \qquad [x_{i,r}^+, x_{j,s}^-]=\delta_{i,j}h_{i,r+s},
\end{equation*}
\begin{equation*}\label{}
[h_{i,r+1}, x_{j,s}^{\pm}]-[h_{i,r}, x_{j,s+1}^{\pm}]=
\pm\frac{1}{2}d_i
a_{ij}\left(h_{i,r}x_{j,s}^{\pm}+x_{j,s}^{\pm}h_{i,r}\right),
\end{equation*}
\begin{equation*}\label{}
[x_{i,r+1}^{\pm}, x_{j,s}^{\pm}]-
[x_{i,r}^{\pm}, x_{j,s+1}^{\pm}]=\pm\frac12
d_ia_{ij}\left(x_{i,r}^{\pm}x_{j,s}^{\pm}
+x_{j,s}^{\pm}x_{i,r}^{\pm}\right),
\end{equation*}
\begin{equation*}\label{}
\sum_\pi
[x_{i,r_{\pi\left(1\right)}}^{\pm},
[x_{i,r_{\pi\left(2\right)}}^{\pm}, \ldots,
[x_{i,r_{\pi\left(m\right)}}^{\pm},
x_{j,s}^{\pm}]\cdots]]=0, i\neq j,
\end{equation*}
for all sequences of non-negative integers $r_1,\ldots,r_m$, where
$m=1-a_{ij}$ and the sum is over all permutations $\pi$ of $\{1,\dots,m\}$.
\end{definition}
For the simplicity, we denote the generators of $\ysl$ by $x_{k}^{\pm}$ and $h_{k}$. In $\yg$, for a fixed $i\in I$, $\{x_{i,r}^{\pm}, h_{i,r}|r\geq 0\}$ generates a subalgebra $Y_i$ of $\yg$ and the assignments $\frac{\sqrt{d_i}}{{d_i}^{k+1}}x_{i,k}^{\pm}\rightarrow x_{k}^{\pm}$ and $\frac{1}{{d_i}^{k+1}}h_{i,k}\rightarrow h_{k}$ define an isomorphism from $Y_i$ to $\ysl$. We denote by $\Delta_{Y_i}$ and $\Delta_{\yg}$ the coproducts of $Y_i$ and $\yg$, respectively.
Let $Y^{\pm}$ and
$H$ be the subalgebras of $\yg$ generated by the $x_{\alpha, k}^{\pm}$, for all positive roots
$\alpha$, and $h_{i, k}$ for all $i\in I$ and $k\in \mathbb{Z}_{\geq 0}$, respectively. Set $N^{\pm}=\sum_{i,k} x_{i, k}^{\pm}Y^{\pm}$. Let $h_i(u)=1+h_{i,0}u^{-1}+h_{i,1}u^{-2}+\ldots$ and $H_{i}(u)=\sum\limits_{k=0}^{\infty} H_{i,k}u^{-k-1}:=\ln\big(h_i(u)\big)$.

\begin{lemma}[Corollary 1.5, \cite{Le}]\label{c7lecor15}
\begin{align*}
[H_{i,k}, x_{j,l}^{\pm}]=&\pm d_ia_{ij}x_{j,l+k}^{\pm}\pm\sum_{\substack{0\leq s\leq k-2\\ k+s\ \text{even}}}2^{s-k}(d_ia_{ij})^{k+1-s}\frac{{ k+1\choose s}}{k+1}x_{j,l+s}^{\pm}.
\end{align*}
\end{lemma}

It has been established that $\yg$ has a Hopf algebra structure. The coproduct of $\yg$, however, is not given explicitly in terms of the generators $x_{i,r}^{\pm}$ and $h_{i,r}$. Chari and Pressley proved the following proposition which is enough for our purpose.
\begin{proposition}[Proposition 2.8, \cite{ChPr4}]\label{Delta}\ \\
Denote $\yg$ by $Y$. Modulo $\overline{Y}\equiv \left(N^{-}Y\otimes YN^{+}\right)\bigcap \left(YN^{-}\otimes N^{+}Y\right)$, we have
\begin{itemize}
  \item [(i)] $\Delta_{\yg}\left(x_{i,k}^{+}\right)\equiv x_{i,k}^{+}\otimes 1+1\otimes x_{i,k}^{+}+\sum\limits_{j=1}^{k}h_{i,j-1}\otimes x_{i,k-j}^{+}$,
  \item [(ii)] $\Delta_{\yg}\left(x_{i,k}^{-}\right)\equiv x_{i,k}^{-}\otimes 1+1\otimes x_{i,k}^{-}+\sum\limits_{j=1}^{k}x_{i,k-j}^{-}\otimes h_{i,j-1}$,
  \item [(iii)] $\Delta_{\yg}\left(h_{i,k}\right)\equiv h_{i,k}\otimes 1+1\otimes h_{i,k}+\sum\limits_{j=1}^{k}h_{i,j-1}\otimes h_{i,k-j}$.
\end{itemize}
\end{proposition}

%\begin{proposition}[Proposition 2.6, \cite{ChPr4}]\label{s1.2.1p26} For any $a\in \C$,
%the assignment $$\tau_{a}(h_{i,k})=\sum_{r=0}^{k}{k\choose r}a^{k-r}h_{i,r},\qquad \tau_{a}(x_{i,k}^{\pm})=\sum_{r=0}^{k}{k\choose r}a^{k-r}x_{i,r}^{\pm}$$ extends to a Hopf algebra automorphism of $\yg$.
%\end{proposition}
We are in the position to introduce the finite-dimensional highest weight representations.
In \cite{Dr2}, Drinfeld gave the following classification of the finite-dimensional irreducible representations of $\yg$.
\begin{theorem}\label{cfdihwr}\
\begin{enumerate}
  \item[(i)] Every irreducible finite-dimensional representation of $\yg$ is highest weight.
  \item[(ii)] The irreducible representation $L\left(\mu\right)$ is finite-dimensional if and only if there exists an $l$-tuple of polynomials $\pi=\big(\pi_1(u),\ldots, \pi_l(u)\big)$ such that $$\frac{\pi_i\left(u+d_i\right)}{\pi_i\left(u\right)}=1+\sum_{k=0}^{\infty} \mu_{i,k}u^{-k-1},$$ in the sense that the right-hand side is the Laurent expansion of the left-hand side about $u=\infty$. The polynomials $\pi_i\left(u\right)$ are called Drinfeld polynomials.
\end{enumerate}
\end{theorem}

%The left dual $^t V$ of $V$ and right dual $V^t$ of $V$ are the representations of $\yg$ on the dual vector space of $V$ defined as follows:
%$\left(y\cdot f\right)\left(v\right)=f\left(S\left(y\right)\cdot v\right)$ and
%$\left(y\cdot g\right)\left(v\right)=g\left(S^{-1}\left(y\right)\cdot v\right)$ respectively,
%where $y\in \yg$, $f\in\ ^tV$, $g\in V^t$, $v\in V$, and $S$ is the antipode of $\yg$.
%It is well known that the dual of an irreducible representation is irreducible.
%The dual spaces $\ ^tV$ and $V^t$ are vital in determining the irreducibility of $V$ as indicated in the proposition below.
%\begin{proposition}[Proposition 3.8, \cite{ChPr8}]\label{VoWWoVhi}%previous version can be found in First draft.tex
%Let $V$ be a finite-dimensional representation of $\yg$. $V$ is irreducible if and only if $V$ and $\ ^tV$ \big(respectively, $V$ and $V^t$\big) are both highest weight $\yg$-modules.
%\end{proposition}

Note that every finite-dimensional highest-weight representation is also a lowest weight representation and vice-versa. The following lemma is an analogue of Lemma 4.2 in \cite{Ch3}. The proof has been omitted as it is similar to the proof of that Lemma.
\begin{lemma}\label{v-w+gvtw}
Let $V$ and $W$ be two finite-dimensional highest weight representations of $\yg$ with lowest and highest weight vectors $v^{-}$ and $w^{+}$, respectively. Then $v^{-}\otimes w^{+}$generates $V\otimes W$.
\end{lemma}

We now turn our attention to the representation theory of $\ysl$. Denote by $V_m(a)$ the finite-dimensional irreducible representation of $\ysl$ associated to the Drinfeld polynomial $\big(u-a\big)\big(u-(a_1+1)\big)\ldots \big(u-(a+m-1)\big)$. $V_m(a)$ is called evaluation representation of $\ysl$; as a $\nysl$-module, $V_m(a)$ is isomorphic to the highest weight $\nysl$-module with highest weight $m$.

\begin{proposition}[Proposition 3.5,\cite{ChPr3}]\label{wra}
For any $m\geq 1$, $V_m(a)$ has a basis \\
$\{w_0,w_1,\ldots, w_m\}$ on which the action of $\ysl$ is given by
\begin{center}
$ x_k^+w_s=(s+a)^k(s+1)w_{s+1},\ \ x_k^-w_s=(s+a-1)^k(m-s+1)w_{s-1},$\\

$h_kw_s=\big((s+a-1)^ks(m-s+1)-(s+a)^k(s+1)(m-s)\big)w_s$.
\end{center}
\end{proposition}

Call that a polynomial $P$ is in general position with respect to another polynomial $Q$ if for any root $t$ of $Q$ and any root $s$ of $P$, we have $t-s\neq 1$.
The following lemma is from paraphrasing Corollary 3.8 \cite{ChPr3}.
\begin{lemma}\label{beautiful}
Let $P_1, P_2,\ldots, P_m$ be polynomials, and let $V(P_i)$ be the irreducible highest weight representation whose Drinfeld polynomial is $P_i$. If $P_j$ is in general position with respect to $P_i$ ($i<j$), then $V(P_1)\otimes V(P_2)\otimes \ldots\otimes V(P_m)$ is a highest weight $\ysl$-module.
\end{lemma}

We close this section by introducing the local Weyl module $W(\pi)$ of $\yg$. It has been established that the category of finite-dimensional
representations of $\yg$ is not semi-simple. Therefore, it is not enough to just study its irreducible representations. In
\cite{ChPr1}, the notion of local Weyl modules for classical and quantum affine algebra was introduced, and we extended this notion for Yangians \cite{TaGu}. The local Weyl module $W(\pi)$ has a nice universal property:  any finite-dimensional highest weight representation of $\yg$ associated to $\pi$ is a quotient of $W(\pi)$.
\begin{definition}
Let $\pi_i\left(u\right)=\prod\limits_{j=1}^{m_i}\left(u-a_{i,j}\right)$ and $\pi=\big(\pi_1(u),\pi_2(u),\ldots, \pi_l(u)\big)$. The local Weyl module $W\left(\mathbf{\pi}\right)$ for $\yg$ is defined as the module generated by a highest weight vector $w_\pi$ that satisfies the following relations: \begin{equation*}
x_{i,k}^{+}w_\pi=0, \qquad \left(x_{i,0}^{-}\right)^{m_i+1}w_\pi=0,\qquad \left(h_{i}\left(u\right)-\frac{\pi_i\left(u+d_i\right)}{\pi_i\left(u\right)}\right)w_\pi=0.\end{equation*}
%\begin{equation}  \end{equation}
\end{definition}
We showed in Theorem 3.8 of \cite{TaGu} that $\operatorname{Dim}\big(W(\pi)\big)\leq \operatorname{Dim}\big(W(\lap)\big)$, where $\lap=\sum\limits_{i=1}^{l} m_i\lambda_i$ and $W(\lap)$ is the local Weyl module of the corresponding current algebra.
The structure of the local Weyl modules of $\yg$ was provided in Theorem 5.23 of \cite{TaGu}, where $\g$ is classical, and Theorem 2 \cite{Ta}, where $\g$ is the simple Lie algebra of type $G_2$. In the end of this paper, we will provide the structure of the local Weyl module of $\yg$, when $\g$ is of type $E$ or $F$.

\section{Braid group action}
 Let $\cal{B}$ be the braid group associated to $\g$, which is the group generated by elements
$T_i$ ($i\in I$) with defining relations:
\begin{align*}
T_iT_j &=T_jT_i,\ \ \text{if}\ \ a_{ij} =0,\\
T_iT_jT_i& =T_jT_iT_j,\ \ \text{if}\ \ a_{ij}a_{ji} =1,\\
(T_iT_j)^2&= (T_jT_i)^2,\ \ \text{if}\ \ a_{ij}a_{ji}=2,\\
(T_iT_j)^3&= (T_jT_i)^3,\ \ \text{if}\  \ a_{ij}a_{ji}
=3,\end{align*} where $i,j\in\{1,2,\cdots ,l\}$. If $w=s_{r_1}s_{r_2}\ldots s_{r_p}$ is a reduced expression of $w$, then denote $T_{r_1}T_{r_2}\ldots T_{r_p}$ by $T_{w}$.

Recall that $\pi_{i,a}$ is the $l$-tuple of polynomials in $u$ such that $\pi_i(u)=u-a$ and $\pi_j(u)=1$. Let $\cal{P}$ be the set of $l$-tuple of rational functions in $u$. Then $\cal{P}$ is a group with generators $\pi_{i,a}$. We now define an action of a braid group of $\g$ on
$\cal{P}$.
\begin{proposition}\label{braidp21} let $P=\big(P_1(u), P_2(u), \cdots, P_l(u)\big)
\in\cal{P}$ and $w=s_{r_1}s_{r_2}\ldots s_{r_p}$ be a reduced expression of $w$. The following formulas define an action of
the braid group $\cal{B}$ of $\g$ on $ \cal{P}$: $T_w(P)=T_{r_1}\Big(T_{r_2}\big(\ldots T_{r_p}(P)\big)\Big)$ and $T_j(P)$ is defined by,
\begin{align*}
\big(T_j(P)\big)_i & =P_i(u),\ \  {\text{if}}\ a_{ij}=0,\\
\big(T_j(P)\big)_i & =P_i(u)P_j(u-\frac{d_i}{2}),\ \  {\text{if}}\
a_{ij}=-1,\\ \big(T_j(P)\big)_i & =P_i(u)P_j(u-1)P_j(u)
,\ \ {\text{if}}\ a_{ij}=-2,\\ \big(T_j(P)\big)_i &=
P_i(u)P_j(u-\frac{3}{2})P_j(u-\frac{1}{2})P_j(u+\frac{1}{2}),\ \
{\text{if}}\ a_{ij}=-3,\\
 (T_j(P))_j&
 =\frac{1}{P_j(u-d_j)}.\end{align*}
\end{proposition}
\begin{proof}
We first claim that it is enough to prove this proposition for $P=\pi_{i,a}$. Let $P$ and $Q$ be two $l$-tuple of rational functions in $u$. %$$PQ=\Big(P_1(u)Q_1(u), P_2(u)Q_2(u),\ldots, P_{l}(u)Q_{l}(u)\Big).$$
Suppose $a_{i,j}=-2$. Then $\Big(T_j(P)\Big)_i=P_i(u)P_j(u-1)P_j(u)$.
\begin{align*}
\big(T_j(PQ)\big)_i & =(PQ)_i(u)(PQ)_j(u-1)(PQ)_j(u)\\
&=\big((P)_i(u)(Q)_i(u)\big)\big((P)_j(u-1)(Q)_j(u-1)\big)\big((P)_j(u)(Q)_j(u)\big)\\
&=\big((P)_i(u)(P)_j(u-1)(P)_j(u)\big)\big((Q)_i(u)(Q)_j(u-1)(Q)_j(u)\big)\\
&=\big(T_j(P)T_j(Q)\big)_i.
\end{align*}
For the other cases, the proofs are similar.

%Therefore, to prove Proposition \ref{braidp21}, it is enough to show that the braid group action is valid when $P=\pi_{i,a}$ for $i\in I$ and $a\in \C$.
Note that $T_m(\pi_{i,a})=\pi_{i,a}$ if $i\neq m$. We next compute $T_i(\pi_{i,a})$. In what following, we may assume that $j=i+1$ and understand $\pi_{l+1,a}=\pi_{0,a}=1$ for any $a\in \C$. Without loss of generality, we use the Cartan matrix as in \cite{Ca}.
We claim:
\begin{enumerate}
  \item[(i)] If $a_{ij}a_{ji} =1$, $T_i(\pi_{i,a})=\pi_{i-1,a+\frac{d_i}{2}}\pi_{i,a+d_i}^{-1}\pi_{i+1,a+\frac{d_i}{2}}$ and $T_{i+1}(\pi_{i+1,a})=\pi_{i,a+\frac{d_i}{2}}\pi_{i,a+d_i}^{-1}\pi$, where either $\pi=\pi_{i+2,a+\frac{d_{i+2}}{2}}$ if $a_{i+2,i+1}=-1$ or $\pi=\pi_{i+2,a+1}\pi_{i+2,a}$ if $a_{i+2,i+1}=-2$.
  \item[(ii)] If $a_{ij}a_{ji} =2$,
  \begin{enumerate}
    \item If $a_{ij}=-1$, then $T_{i}(\pi_{i,a})=\pi_{i-1,a+\frac{d_{i-1}}{2}}\pi_{i,a+2}^{-1}\pi_{i+1,a+1}\pi_{i+1,a}$ and $T_j(\pi_{j,a})=\pi_{j-1,a+\frac{d_i}{2}}\pi_{j,a+1}^{-1}\pi_{j+1,a+\frac{d_{j+1}}{2}}$.
    \item If $a_{ij}=-2$, then the only case is $\g$ is of type $C$ and $i=l-1$ and $j=l$. $T_{l-1}(\pi_{l-1,a})=\pi_{l-2,a+\frac{1}{2}}\pi_{l-1,a+1}^{-1}\pi_{l,a+1}$ and $T_{l}(\pi_{l,a})=\pi_{l-1,a+1}\pi_{l-1,a}\pi_{l,a+2}^{-1}$.
  \end{enumerate}
  \item [(iii)] If $a_{ij}a_{ji} =3$, then $\g=G_2$, $i=1$ and $j=2$. $T_1(\pi_{1,a})=\pi_{1,a+3}^{-1}\pi_{2,a+\frac{3}{2}}\pi_{2,a+\frac{1}{2}}\pi_{2,a-\frac{1}{2}}$ and $T_2(\pi_{2,a})=\pi_{1,a+\frac{3}{2}}\pi_{2,a+1}^{-1}$.
\end{enumerate}
Then it is easy to check that this action is well defined with respect to the defining relations of the Braid group; we omit the details.
\end{proof}
%We close this section by establishing some notations. Let $w=s_{r_1}s_{r_2}\ldots s_{r_p}$ is a reduced expression of $w$. %We define $T_{w}=T_{r_1}T_{r_2}\ldots T_{r_p}$ and $T_{w}(\pi)=T_{r_1}T_{r_2}\ldots T_{r_p}(\pi)$.
\section{Main results}
%Denote by $S_{r}(a)$ the string of complex numbers $\{a, a+1,\ldots, a+r-1\}$ for any $a\in\C$ and $r\in \N$. Two strings
%$S_{p}(a)$ and $S_{q}(b)$ are said to be in general position if $S_{p}(a)\subseteq S_{q}(b)$, $S_{q}(b)\subseteq S_{p}(a)$ or
%$S_{p}(a)\bigcap S_{q}(b)=S_{p}(a+1)\bigcap S_{q}(b)=S_{p}(a-1)\bigcap S_{q}(b)=\varnothing$ \cite{ChPr3}. Then a polynomial $P(u)$ can be written uniquely, up to reorder the factors, as a product
%$$P(u)=\prod\limits_{r=1}^{k}\big(u-a_r\big)\big(u-(a_r+1)\big)\ldots \big(u-(a_r+m_r-1)\big),$$ where $S_{m_p}(a_p)$ and $S_{m_q}(a_q)$ are in general position for any $p, q\in
%\{1,\ldots,k\}$. Let $S(P)$ be the collection of the pairs $(a_r, m_r)$, $1\leq r\leq k$, defined above. Another polynomial $Q(u)=\prod\limits_{r=1}^{s}\big(u-a'_r\big)\big(u-(a'_r+1)\big)\ldots \big(u-(a'_r+m'_r-1)\big)$ is in general position with respect to $P(u)$ if $a-a'_p\neq 1+k$ for $p\in \{1,\ldots,s\}$, $q\in\{1,\ldots,k\}$ and $0\leq k\leq m'_p-1$. %We call polynomial $Q(u)$ is in general $d$-position with respect to $P(u)$ if $a_q-a'_p\neq d+k$ for $p\in \{1,\ldots,s\}$, $q\in\{1,\ldots,k\}$ and $0\leq k\leq m'_p-1$.

The subalgebra generated by $x_{i,0}^{\pm}$ and $h_{i,0}$, $i\in I$, is isomorphic to the universal enveloping algebra
$U(\g)$, hence $V:=V(\pi)$ can be viewed as a $U(\g)$-module with highest weight $\lambda_{\pi}=\sum\limits_{i\in I} \operatorname{Deg}(\pi_i)\omega_i$.  The weight subspace $V_{\lambda_\pi}$ is one-dimensional and spanned
by a highest weight vector $v^{+}$ of $V$. Let $w=s_{r_1}s_{r_2}\ldots s_{r_p}$ be a reduced expression of $w$ in the Weyl group $\W$ of $\g$, where $s_{r_i}$ be a simple reflection. Suppose $s_{r_{j+1}}s_{r_{j+2}}\ldots s_{r_p}(\lap)=m_j\omega_{r_j}+\sum\limits_{n\neq r_j}c_n\omega_n$. Then set
\begin{center}
$v_{w(\lap)}=\left(x_{r_1,0}^{-}\right)^{m_{1}}\left(x_{r_2,0}^{-}\right)^{m_{2}}\ldots \left(x_{r_p,0}^{-}\right)^{m_{p}}v^{+}.$
\end{center}
In particular, when $w=w_0$, $v_{w_0(\lap)}=v^{-}$, where $w_0$ is the longest element in $\W$ and $v^{-}$ is a lowest weight vector of $V$.
It is well known that the weight space $V_{w(\lambda_\pi)}$ of weight
$w(\lambda_\pi)$ is one dimensional spanned by an extremal vector $v_{w(\lambda_\pi)}$ and $h_{i,k}$ acts on $v_{w(\lambda_\pi)}$ by a scalar $\mu_{ik,w}$ \cite{ChPr4}. Denote by $\mu_{i,w}(u)$ the formal series $1+\sum\limits_{k\in \Z_{\geq 0}}\mu_{ik,w}u^{-k-1}$ and set $\Delta_{w}^{+}=\{\alpha\in \Delta^{+}: w^{-1}(\alpha)\in -\Delta^{+}\}$, where $\Delta^{+}$ is denoted the set of positive root as the convention.
\begin{lemma}[Remark 3.1, \cite{ChPr4}]\label{muuiaqopnl}There exists a polynomial $P_{i,w}(u)$ such that
%Let $w$ be any element of the Weyl group of $\g$ and set $\Delta_w^{+}=\{\alpha\in \Delta^{+}: w^{-1}(\alpha)\in -\Delta^{+}\}$.
\begin{center}
  $\mu_{i,w}(u)=\left\{
                                                           \begin{array}{ll}
                                                             \frac{P_{i,w}(u+d_i)}{P_{i,w}(u)}, & \hbox{if $\alpha_i\notin \Delta_{w}^{+}$;} \\
                                                             \frac{P_{i,w}(u)}{P_{i,w}(u+d_i)}, & \hbox{if $\alpha_i\in \Delta_{w}^{+}$.}
                                                           \end{array}
                                                         \right.
  $
\end{center}
\end{lemma}

Let $V$ be a finite-dimensional highest weight representation of $\ysl$ whose associated polynomial is $\pi$. Let $v^{+}$ and $v^{-}$ be highest and lowest weight vectors of $V$, respectively.
%\begin{corollary}[Proposition 2.4 \& Corollary 2.5, \cite{ChPr3}]
%$\ ^tV\cong  V(-1)$ and $V^t\cong V(1)$. Moreover, their Drinfeld polynomials are $\pi(u+1)$ and $\pi(u-1)$, respectively.
%\end{corollary}
It follows from Lemma \ref{muuiaqopnl} that
\begin{corollary}
$h(u)v^{-}=\frac{\pi(u-1)}{\pi(u)}v^{-}.$
\end{corollary}

Define $\sigma_j:=s_{r_{j+1}}s_{r_{j+2}}\ldots s_{r_p}$ and $v_{\sigma_j(\lap)}:=\left(x_{r_{j+1},0}^{-}\right)^{m_{j+1}}\left(x_{r_{j+2},0}^{-}\right)^{m_{j+2}}\ldots \left(x_{r_p,0}^{-}\right)^{m_{p}}v^{+}$.
\begin{lemma}\label{S5L1yiahwr}
$Y_{r_j}\left(v_{\sigma_j(\lap)}\right)$ is a highest weight representation of $Y_{r_j}$.
\end{lemma}
\begin{proof}
Suppose that $x_{r_j,0}^{+}v_{\sigma_j(\lap)}\neq 0$,
then $\sigma_j(\lap)+\alpha_{r_j}$ is a weight of $V$. Thus $\lap+\sigma_{j}^{-1}(\alpha_{r_j})$ is a weight which does not precede $\lap$, which implies that $\sigma_{j}^{-1}(\alpha_{r_j})$ is a negative root.
However, $\ell(s_{r_j}\sigma_j)=\ell(\sigma_j)+1$,  hence $\sigma_{j}^{-1}(\alpha_{r_j})$ is a positive root, so we have a contradiction.
Thus $x_{r_j,0}^{+}v_{\sigma_j(\lap)}=0$.
Since the weight space of weight $\sigma_j(\lap)$ is 1-dimensional and $H$ is a commutative subalgebra of $\yg$, $h_{r_j,s}v_{\sigma_j(\lap)}=c_{j,s}v_{\sigma_j(\lap)}$, where $c_{j,s}\in \C$.
Thus $Y_{r_j}\left(v_{\sigma_j(\lap)}\right)$ is a highest weight representation of $Y_{r_j}$.
\end{proof}

The coming lemma was proved in \cite{TaGu} and \cite{Ta}.
 \begin{lemma}\label{P3S3L34TGT}
 Let $v^{+}$ be a highest weight representation of $j$-th fundamental representation $W_a(\omega_j)$ of $\yg$. The associated polynomial $P_{i, s_j}(u)$ is listed as following:
\begin{enumerate}
  \item[(i)] $P_{i, s_j}(u)=1$ if $a_{ij}=0$.
  \item[(ii)] $P_{i, s_j}(u)=u-(a+\frac{d_i}{2})$ if $a_{ij}=-1$.
  \item[(iii)] $P_{i, s_j}(u)=\big(u-(a+1)\big)\big(u-a\big)$ if $a_{ij}=-2$.
  \item[(iv)] $P_{i, s_j}(u)=\big(u-(a+\frac{3}{2})\big)\big(u-(a+\frac{1}{2})\big)\big(u-(a-\frac{1}{2})\big)$ if $a_{ij}=-3$.
\end{enumerate}
\end{lemma}
Let $l(w)$ be the length of $w\in\W$. Suppose $l(s_iw)=l(w)+1$. It follows from Lemma \ref{S5L1yiahwr} that $h_{i}(u)v_{w(\lap)}=\mu_{i,w}(u)v_{w(\lap)}$.  In the next proposition, we show that the eigenvalues $\mu_{i,w}(u)$ is determined by the braid group action on $\pi$.
\begin{proposition}\label{ogpigf3p}
$\mu_{i,w}(u)=\frac{\big(T_{w}(\pi)\big)_i(u+d_i)}{\big(T_{w}(\pi)\big)_i(u)}$.
\end{proposition}
\begin{proof}
We proceed by induction on $l(w)$. If $l(w)=0$, the claim is true since $T_{w}(\pi)=\pi$. Suppose that $l(w)=1$, say $w=s_j$. Since the weight space $V_{s_j(\lambda_{\pi})}$ is 1-dimensional, we may assume that $v_{s_j(\lap)}=(x_{j,0}^{-})^{m_j}v^{+}\neq 0$. We are going to compute the eigenvalues $\mu_{ik,w}$ of $h_{i,k}$ on the weight vector $(x_{j,0}^{-})^{m_j}v^{+}$. Recall that $H_{i}(u)=\sum\limits_{k=0}^{\infty} H_{i,k}u^{-k-1}:=\ln\big(h_i(u)\big)$.
%\begin{align}
%H_{i,k}(x_{j,0}^{-})^{m_j}v^{+}&= (x_{j,0}^{-})^{m_j}H_{i,k}v^{+}+[H_{i,k},(x_{j,0}^{-})^{m_j}]v^{+}\nonumber\\
%&= (x_{j,0}^{-})^{m_j}H_{i,k}v^{+}+\sum_{r=0}^{m_j-1}(x_{j,0}^{-})^{r}[H_{i,k},x_{j,0}^{-}](x_{j,0}^{-})^{m_j-r-1}v^{+}.
%\end{align}
Denote $\sum\limits_{r=0}^{m_j-1}(x_{j,0}^{-})^{r}x_{j,s}^{-}(x_{j,0}^{-})^{m_j-r-1}v^{+}$ by $E_{s}$ and $\sum\limits_{r=0}^{m_j-1}(x_{j,0}^{-})^{r}[H_{i,k},x_{j,0}^{-}](x_{j,0}^{-})^{m_j-r-1}v^{+}$ by $F_{k}$. %It follows from Lemma \ref{c7lecor15} that $[H_{i,k},x_{j,0}^{-}]=\sum\limits_{s=0}^{k}a_sx_{j,s}^{-}$. Then $F_{i,k}=\sum\limits_{s=0}^{k}a_s E_{s}$.
%On the other hand,
Note that \begin{align*}
H_{j,k}(x_{j,0}^{-})^{m_j}v^{+}&= (x_{j,0}^{-})^{m_j}H_{j,k}v^{+}+[H_{j,k},(x_{j,0}^{-})^{m_j}]v^{+}\\
&=(x_{j,0}^{-})^{m_j}H_{j,k}v^{+}+\sum_{r=0}^{m_j-1}(x_{j,0}^{-})^{r}[H_{j,k},x_{j,0}^{-}](x_{j,0}^{-})^{m_j-r-1}v^{+}.
\end{align*}
Therefore,
$\sum\limits_{r=0}^{m_j-1}(x_{j,0}^{-})^{r}[H_{j,k},x_{j,0}^{-}](x_{j,0}^{-})^{m_j-r-1}v^{+}
=H_{j,k}(x_{j,0}^{-})^{m_j}v^{+}-(x_{j,0}^{-})^{m_j}H_{j,k}v^{+}.$
%It is clear that $H_{j,k}(x_{j,0}^{-})^{m_j}v^{+}-(x_{j,0}^{-})^{m_j}H_{j,k}v^{+}$ is a scalar of $(x_{j,0}^{-})^{m_j}v^{+}$.
Denote $H_{j,k}(x_{j,0}^{-})^{m_j}v^{+}-(x_{j,0}^{-})^{m_j}H_{j,k}v^{+}$ by $G_k$.

We show next that $F_{k}=\sum\limits_{t=0}^{k}c_tG_{t}$ and each $c_t$ does not depend on $m_j$. By Lemma
\ref{c7lecor15}, suppose that
$[H_{i,k},x_{j,0}^{-}]=\sum\limits_{s=0}^{k}a_{k,s}x_{j,s}^{-}$ and $[H_{j,k},x_{j,0}^{-}]=\sum\limits_{s=0}^{k}b_{k,s}x_{j,s}^{-}$, where $b_{k,k}\neq 0$. Then $F_{k}=\sum\limits_{s=0}^{k}a_{k,s} E_{s}$ and $G_{k}=\sum\limits_{s=0}^{k}b_{k,s}E_{s}$. We note that the values of $a_{k,s}$ and $b_{k,s}$ do not depend on $m_j$. It is clear that
$E_{s}=\sum\limits_{t=0}^{s}b_{s, t}'G_t$. %, where $b_{s,t}'$ only depends on $b_{t,r}$. %, where $r\leq t\leq s$.
 In particular, the values of $b_{s,t}'$ do not depend on $m_j$. $$F_{k}=\sum\limits_{s=0}^{k}a_{k,s} E_{s}=\sum\limits_{s=0}^{k}a_{k,s}
\big(\sum\limits_{t=0}^{s}b_{s,t}'G_t\big)=\sum\limits_{t=0}^{k}c_tG_{t}.$$ %Recall that both $a_{k,s}$ and $b_{s,t}'$ does not depends on the value of $m_j$.
Since the coefficient $c_t$ is only depends on both $a_{k,s}$ and $b_{s,t}'$,
the coefficient $c_t$ does not change the values with respect to $m_j$ (the exact values of $c_t$ are irrelevant here).

Suppose $a_{ij}=-1$. Consider the case when $V=W_{a}(\omega_j)$ in this paragraph. In this case, $\pi_j(u)=u-a$ and $m_j=1$. By
Lemma \ref{P3S3L34TGT}, $$[h_{i}(u),x_{j,0}^{-}]v^{+}=h_{i}(u)x_{j,0}^{-}v^{+}=\frac{1-(a+\frac{d_i}{2}-d_i)u^{-1}}{1-(a+\frac{d_i}{2})u^{-1}}x_{j,0}^{-}v^{+}.$$ %when $a_{ij}=-2$,
%$$[h_{i}(u),x_{j,0}^{-}]v^{+}=h_{i}(u)x_{j,0}^{-}v^{+}=\frac{(1-au^{-1})(1-(a-1)u^{-1})}{(1-(a+1)u^{-1})(1-au^{-1})}x_{j,0}^{-}v^{+}.$$
Recall that $H_i(u)=\ln(h_i(u))$. Then
$$[H_{i}(u),x_{j,0}^{-}]v^{+}=\Big(\ln\big(1-(a-\frac{d_i}{2})u^{-1}\big)-\ln\big(1-(a+\frac{d_i}{2})u^{-1}\big)\Big)x_{j,0}^{-}v^{+}.$$
Therefore, by comparing the coefficient of $u^{-k-1}$ both sides of the above equality, $$F_k=[H_{i,k},x_{j,0}^{-}]v^{+}=\frac{1}{k+1}\big((a+\frac{d_i}{2})^{k+1}-(a-\frac{d_i}{2})^{k+1}\big)x_{j,0}^{-}v^{+}.$$ %where the coefficient of $u^{-k-1}$ in $\ln\Big(\frac{1-(a+\frac{d_i}{2})u^{-1}}{1-(a+d_i+\frac{d_i}{2})u^{-1}}\Big)$.
On the other hand, we have $h_{j}(u)v^{+}=\frac{u-(a-d_j)}{u-a}v^{+}$ and
$h_{j}(u)x_{j,0}^{-}v^{+}=\frac{u-(a+d_j)}{u-a}x_{j,0}^{-}v^{+}$. Then
$$G_{t}=H_{j,t}x_{j,0}^{-}v^{+}-x_{j,0}^{-}H_{j,t}v^{+}=\frac{1}{t+1}\big((a-d_j)^{t+1}-(a+d_j)^{t+1}\big)x_{j,0}^{-}v^{+}.$$
Since $F_{k}=\sum\limits_{t=0}^{k}c_tG_{t}$, by comparing their coefficients of $x_{j,0}^{-}v^{+}$ both sides, we obtain:
\begin{equation}\label{eq41}
\frac{1}{k+1}\big((a+\frac{d_i}{2})^{k+1}-(a-\frac{d_i}{2})^{k+1}\big)
=\sum_{t=0}^{k}\frac{c_t}{t+1}\big((a-d_j)^{t+1}-(a+d_j)^{t+1}\big).
\end{equation}
%which is the coefficient of $u^{-k-1}$ in $\ln\Big(\frac{1-(a+\frac{d_i}{2})u^{-1}}{1-(a+d_i+\frac{d_i}{2})u^{-1}}\Big)$.

Now, suppose $\pi_j(u)=\prod\limits_{s=1}^{m_j}(u-a_s)$ for $m_j\geq 2$. Thus $h_{j}(u)v^{+}=\prod\limits_{s=1}^{m_j}\frac{u-(a_s-d_j)}{u-a_s}v^{+}$ and
$h_{j}(u)\big(x_{j,0}^{-}\big)^{m_j}v^{+}=\prod\limits_{s=1}^{m_j}\frac{u-(a_s+d_j)}{u-a_s}\big(x_{j,0}^{-}\big)^{m_j}v^{+}$,
then $$H_{j}(u)v^{+}=\sum\limits_{s=1}^{m_j}\Big(\ln\big(1-(a_s-d_j)u^{-1}\big)-\ln\big(1-a_su^{-1}\big)\Big)v^{+}$$
and $$H_{j}(u)(x_{j,0}^{-})^{m_j}v^{+}=\sum\limits_{s=1}^{m_j}\Big(\ln\big(1-(a_s+d_j)u^{-1}\big)
-\ln\big(1-a_su^{-1}\big)\Big)(x_{j,0}^{-})^{m_j}v^{+}.$$ Therefore,
$$G_{t}=H_{j,k}(x_{j,0}^{-})^{m_j}v^{+}-(x_{j,0}^{-})^{m_j}H_{j,k}v^{+}=\sum\limits_{s=1}^{m_j}\frac{1}{t+1}\big((a_s-d_j)^{t+1}-(a_s+d_j)^{t+1}\big)(x_{j,0}^{-})^{m_j}v^{+}.$$
\begin{align*}
F_{k}&=\sum_{t=0}^{k}c_tG_{t}\\
       &=\sum_{t=0}^{k}\Big(\frac{c_t}{t+1}\sum_{s=1}^{m_j}\big((a_s-d_j)^{t+1}-(a_s+d_j)^{t+1}\big)\Big)(x_{j,0}^{-})^{m_j}v^{+}\\
       &=\sum_{s=1}^{m_j}\Big(\sum_{t=0}^{k}\frac{c_t}{t+1}\big((a_s-d_j)^{t+1}-(a_s+d_j)^{t+1}\big)\Big)(x_{j,0}^{-})^{m_j}v^{+}\\
       &=\sum_{s=1}^{m_j}\frac{1}{k+1}\big((a_s+\frac{d_i}{2})^{k+1}-(a_s-\frac{d_i}{2})^{k+1}\big)(x_{j,0}^{-})^{m_j}v^{+}. \qquad \text{\big(by (\ref{eq41})\big)}
\end{align*}
Note that the coefficient of $(x_{j,0}^{-})^{m_j}v^{+}$ in the above computation equals the coefficient of $u^{-k-1}$ in $\sum\limits_{s=1}^{m_j}\ln\Big(\frac{1-(a_s-\frac{d_i}{2})u^{-1}}{1-(a_s+\frac{d_i}{2})u^{-1}}\Big)$. Thus
$$\sum_{r=0}^{m_j}(x_{j,0}^{-})^{r}[H_{i}(u),x_{j,0}^{-}](x_{j,0}^{-})^{m_j-r-1}v^{+}
=\ln\Big(\prod\limits_{s=1}^{m_j}\frac{1-(a_s-\frac{d_i}{2})u^{-1}}{1-(a_s+\frac{d_i}{2})u^{-1}}\Big)
(x_{j,0}^{-})^{m_j}v^{+}.$$
%The coefficient of the right side of the above equality equals $\ln\Big(\frac{P_j(u-\frac{d_i}{2}+d_i)}{P_j(u-\frac{d_i}{2})}\Big)$.
Thus we have that
\begin{align*}
H_{i}(u)(x_{j,0}^{-})^{m_j}v^{+}&=
\sum_{r=0}^{m_j}(x_{j,0}^{-})^{r}[H_{i}(u),x_{j,0}^{-}](x_{j,0}^{-})^{m_j-r-1}v^{+}+(x_{j,0}^{-})^{m_j}H_{i}(u)v^{+}\\
&=\ln\Big(\frac{P_j(u-\frac{d_i}{2}+d_i)}{P_j(u-\frac{d_i}{2})}\Big)(x_{j,0}^{-})^{m_j}v^{+}+
\ln\Big(\frac{P_i(u+d_i)}{P_i(u)}\Big)(x_{j,0}^{-})^{m_j}v^{+}\\
&=\ln\Big(\frac{P_j(u-\frac{d_i}{2}+d_i)}{P_j(u-\frac{d_i}{2})}\cdot\frac{P_i(u+d_i)}{P_i(u)}\Big)(x_{j,0}^{-})^{m_j}v^{+}.
\end{align*}
Recall that $\big(T_{j}(P)\big)_i(u)=P_i(u)P_j(u-\frac{d_i}{2})$. So $$h_{i}(u)(x_{j,0}^{-})^{m_j}v^{+}=
\frac{P_j(u-\frac{d_i}{2}+d_i)}{P_j(u-\frac{d_i}{2})}\cdot\frac{P_i(u+d_i)}{P_i(u)}(x_{j,0}^{-})^{m_j}v^{+}=
\frac{\big(T_{j}(P)\big)_i(u+d_i)}{\big(T_{j}(P)\big)_i(u)}(x_{j,0}^{-})^{m_j}v^{+},$$
i.e., $\mu_{i,w}(u)=\frac{\big(T_{w}(P)\big)_i(u+d_i)}{\big(T_{w}(P)\big)_i(u)}$.

Similar discussions can be carried on when $a_{ij}=0$, $-2$ or $-3$. We omit the details.

Suppose $l(w)=n\geq 2$ and $w=s_kw'$ with $l(w')=n-1$. Note that $T_{w}=T_{k}T_{w'}$, $T_{w}(\pi)=T_{k}\big(T_{w'}(\pi)\big)$ and $v_{w(\lap)}=(x_{k,0}^{-})^{m_k}v_{w'(\lap)}$ for some $m_k\in \Z_{\geq 0}$. This proof is almost the same as the above paragraphs by replacing $v^{+}$ by $v_{w'(\lap)}$. By induction, this completes the proof of this proposition.
\end{proof}
\begin{remark}
Suppose that $\big(T_{\sigma_j}(\pi)\big)_{r_j}(u)=\prod\limits_{r=1}^{n}\prod\limits_{s_r=0}^{m_r-1}\big(u-(a_r+s_r)\big)$.
Since one need to normalize the generators of $Y_{r_j}$ to satisfy the defining relations of $\ysl$, the associated polynomial of $Y_{r_j}\left(v_{\sigma_j(\lambda_{\pi^{(m)}})}\right)$ as a $\ysl$-module is $\prod\limits_{r=1}^{n}\prod\limits_{s_r=0}^{m_r-1}(u-\frac{a_r+s_r}{d_i})$, which is denoted by $\big(\widetilde{T_{w}(\pi)}\big)_{r_j}$.
\end{remark}

Before showing the main theorem of this paper, we need the following lemma first.
\begin{lemma}[Lemma 4.3, \cite{ChPr4}]\label{saytiirr}
Let $V$ be an finite-dimensional irreducible representation of $\yg$ with highest weight vector $v^{+}$, and let $\mathfrak{t}$ be a diagram subalgebra of $\g$. Then $Y(\mathfrak{t})$-subrepresentation of $V$ generated by $v^{+}$ is irreducible.
\end{lemma}
%For a fixed $i\in I$, $\operatorname{span}\{x_{i,r}^{\pm}, h_{i,r}\}$ for $r\in \Z_{\geq 0}$ is isomorphic to $\ysl$. However, these generators does not satisfy the defining relations if $d_i\neq 1$. Therefore, we need to rescale these generators if we want to use results of $\ysl$. Let
%\begin{center}
%$\widetilde{x}_{i,r}^{\pm}=\frac{\sqrt{d_i}}{d_i^{r+1}}x_{i,r}^{\pm}$ and $\widetilde{h}_{i,r}=\frac{1}{d_i^{r+1}}h_{i,r}$.
%\end{center}
%Then $\operatorname{span}\{\widetilde{x}_{i,r}^{\pm}, \widetilde{h}_{i,r}\}$ is isomorphic to $\ysl$ and these generators satisfy the defining relations.
From now on, we assume $s_{r_1}s_{r_2}\ldots s_{r_p}$ is a reduced expression of the longest element $w_0\in \W$.
%Recall that $\sigma_j:=s_{r_{j+1}}s_{r_{j+2}}\ldots s_{r_p}$ and $v_{\sigma_j(\lap)}:=\left(x_{r_{j+1},0}^{-}\right)^{m_{j+1}}\left(x_{r_{j+2},0}^{-}\right)^{m_{j+2}}\ldots \left(x_{r_p,0}^{-}\right)^{m_{p}}v^{+}$, where $v^{+}$ is a highest weight vector in a finite-dimensional irreducible representation $V(\pi)$.
\begin{theorem}\label{mtimtp}
Let $\pi^{(i)}$, $1\leq i\leq k$, be $l$-tuples of polynomials in $u$ and $V(\pi^{{(i)}})$ be the finite-dimensional irreducible representations of $\yg$ whose Drinfeld polynomials are $\pi^{(i)}$. The ordered tensor product $L=V(\pi^{(1)})\otimes V(\pi^{(2)})\otimes\ldots\otimes V(\pi^{(k)})$ is a highest weight representation if the polynomial $\widetilde{\big(T_{\sigma_j}(\pi^{(m)})\big)}_{r_j}$ is in general position with respect to $(\widetilde{\pi^{(n)}})_{r_j}$ for all $1\leq j\leq p$ and $1\leq m<n\leq k$.
\end{theorem}
\begin{proof}
Let $v_m^{+}$ be the highest weight vector of $V(\pi^{(m)})$ and let $v_1^{-}$ be the lowest weight vector of $V(\pi^{(1)})$. Recall that $Y_{r_j}\left(v_{\sigma_j(\lambda_{\pi^{(m)}})}\right)$ is a highest weight representation of $\ysl$. Thus $Q:=\big(\widetilde{T_{w}(\pi^{(m)}})\big)_{r_j}$ is a polynomial.

We prove this theorem by induction on $k$.
For $k=1$, $L=V(\pi^{(1)})$ is irreducible and hence it is a highest weight representation. We assume that the claim is true for all positive integers less than or equal to $k-1$ ($k\geq 2$). By the induction hypothesis, $V(\pi^{(2)})\otimes\ldots\otimes V(\pi^{(k)})$ is a highest weight representation of $\yg$ and its highest weight vector is $v^{+}=v_2^{+}\otimes \ldots\otimes v_k^{+}$. To show that $L$ is a highest weight representation, it suffices to show $v^{-}_{1}\otimes v^{+}\in \yg\left(v^{+}_1\otimes v^{+}\right)$ by Lemma \ref{v-w+gvtw}.

We first show that $$Y_{r_j}\left(v_{\sigma_j(\lambda_{\pi^{(1)}})}\right)\otimes Y_{r_j}\left(v_2^{+}\right)\otimes\ldots\otimes Y_{r_j}\left(v_k^{+}\right)= Y_{r_j}\left(v_{\sigma_j(\lambda_{\pi^{(1)}})}\otimes v_2^{+}\otimes\ldots \otimes v_k^{+}\right).$$
For $2\leq n\leq k$, as $\ysl$ module, $Y_{r_j}(v_n^{+})$ is isomorphic to $V\Big(\big(\widetilde{\pi^{(n)}}\big)_{r_j}\Big)$ by Lemma \ref{saytiirr}. Since the polynomial $Q$ is in general position with respect to $(\widetilde{\pi^{(n)}})_{r_j}$, $W(Q)\otimes Y_{r_j}\left(v_2^{+}\right)\otimes\ldots\otimes Y_{r_j}\left(v_k^{+}\right)$ is a highest weight representation of $\ysl$ by Lemma \ref{beautiful}.
Hence its quotient $Y_{r_j}\left(v_{\sigma_j(\lambda_{\pi^{(1)}})}\right)\otimes Y_{r_j}\left(v_2^{+}\right)\otimes\ldots\otimes Y_{r_j}\left(v_k^{+}\right)$ is a highest weight module of $Y_{r_j}$ with highest weight vector $v_{\sigma_j(\lambda_{\pi^{(1)}})}\otimes v_2^{+}\otimes\ldots \otimes v_k^{+}$. Thus $$Y_{r_j}\left(v_{\sigma_j(\lambda_{\pi^{(1)}})}\otimes v_2^{+}\otimes\ldots \otimes v_k^{+}\right)\supseteq Y_{r_j}\left(v_{\sigma_j(\lambda_{\pi^{(1)}})}\right)\otimes Y_{r_j}\left(v_2^{+}\right)\otimes\ldots\otimes Y_{r_j}\left(v_k^{+}\right).$$
Since $\overline{Y}$ annihilates $v_{\sigma_j(\lambda_{\pi^{(1)}})}\otimes v_2^{+}\otimes\ldots \otimes v_k^{+}$, by Proposition \ref{Delta} that
$$\Delta_{\yg}(x)\left(v_{\sigma_j(\lambda_{\pi^{(1)}})}\otimes v_2^{+}\otimes\ldots \otimes v_k^{+}\right)=\Delta_{\ysl}(x)\left(v_{\sigma_j(\lambda_{\pi^{(1)}})}\otimes v_2^{+}\otimes\ldots \otimes v_k^{+}\right),$$ where $x\in Y_{r_j}$. Then it is easy to see that
$$Y_{r_j}\left(v_{\sigma_j(\lambda_{\pi^{(1)}})}\otimes v_2^{+}\otimes\ldots \otimes v_k^{+}\right)\subseteq Y_{r_j}\left(v_{\sigma_j(\lambda_{\pi^{(1)}})}\right)\otimes Y_{r_j}\left(v_2^{+}\right)\otimes\ldots\otimes Y_{r_j}\left(v_k^{+}\right).$$ Therefore the claim is true.

Since $v_{\sigma_{j-1}(\lambda_{\pi^{(1)}})}\otimes v^{+}=(x_{r_j,0}^{-})^{m_j}v_{\sigma_j(\lambda_{\pi^{(1)}})}\otimes v^{+}\in Y_{r_j}\left(v_{\sigma_{j}(\lap)}\right)\otimes Y_{r_j}\left(v^{+}\right)$, we have $v_{\sigma_{j-1}(\lap)}\otimes v^{+}\in Y_{r_j}\left(v_{\sigma_j(\lambda_{\pi^{(1)}})}\otimes v^{+}\right)$.
By downward induction on the subscript $j$ of $v_{\sigma_{j}(\lap)}$,
$v^{-}_{1}\otimes v^{+}\in \yg\left(v^{+}_1\otimes v^{+}\right)$. By Lemma \ref{v-w+gvtw}, $L=\yg \left(v_1^{+}\otimes v^{+}\right)$ is a highest weight representation.
\end{proof}

\section{The cyclicity condition made explicit}

With a specific reduced expression of $w_0$, one can compute explicitly the polynomial $\big(T_{r_{j+1}}T_{r_{j+2}}\ldots T_{r_p}(\pi^{(m)})\big)_{r_j}$ via the braid group action defined in Section 3. Then a cyclicity condition for the tensor product $L$ can be obtained immediately by using Theorem \ref{mtimtp}. In this section, we provide a finite set of numbers at which a tensor product of Kirillov-Reshetikhin modules may fail to be cyclic. At the end of this section, we give the structure of local Weyl modules when $\g$ is either of type $E$ or of type $F$.

When the simple Lie algebra $\g$ is classical, let the nodes of the Dynkin diagram of $\g$ be numbered as in \cite{BeKaOhPa} and use the reduced expression of $w_0$ therein. When $\g$ is exceptional, the nodes of the Dynkin diagram of the exceptional simple Lie algebras $\g$ be numbered as following.

{\centering
$E_6$: $\DynkinEs$;\quad $E_7$: $\DynkinEsv$;\newline $E_8$: $\DynkinEe$;\qquad\qquad $F_4$: \DynkinF;\qquad $G_2$: \DynkinG.}\\

One reduced expression for the longest element $w_0$ is listed as follows.
\begin{enumerate}
  \item[(a)] $E_6$: $u_6=s_{1}s_{2}s_{3}s_{1}s_{4}s_{2}s_{3}s_{1}s_{4}s_{3}s_{5}s_{4}s_{2}s_{3}s_{1}s_{4}s_{3}s_{5}s_{4}s_{2}s_{6}s_{5}s_{4}s_{2}s_{3}s_{1}s_{4}s_{3}s_{5}s_{4}s_{2}s_{6}s_{5}s_{4}s_{3}s_{1}
$.
  \item[(b)] $E_7$: $u_7=u_{6}s_{7}s_{6}s_{5}s_{4}s_{2}s_{3}s_{1}s_{4}s_{3}s_{5}s_{4}s_{2}s_{6}s_{5}s_{4}s_{3}
s_{1}s_{7}s_{6}s_{5}s_{4}s_{2}s_{3}s_{4}s_{5}s_{6}s_{7}$.
  \item [(c)] $E_8$: $u_{7}s_{8}s_{7}s_{6}s_{5}s_{4}s_{2}s_{3}s_{1}s_{4}s_{3}s_{5}s_{4}s_{2}s_{6}s_{5}s_{4}s_{3}s_{1}s_{7}s_{6}s_{5}s_{4}s_{2}s_{3}s_{4}s_{5}
s_{6}s_{7}s_{8}s_{7}s_{6}s_{5}s_{4}s_{2} s_{3}s_{1}\\ \ \qquad\times s_{4}s_{3}s_{5}s_{4}s_{2} s_{6}s_{5}s_{4}s_{3}s_{1}s_{7}s_{6}s_{5}s_{4}s_{2}s_{3}
s_{4}s_{5}s_{6}s_{7}s_{8}$.%, where $u_{7}$ is the reduced expression for the longest element in the Weyl group of $E_7$.
  \item [(d)] $F_4$: $s_1s_2s_1s_3s_2s_1s_3s_2s_3s_4s_3s_2s_1s_3s_2s_3s_4s_3s_2s_1s_3s_2s_3s_4$.
  \item [(e)] $G_2$: $s_2s_1s_2s_1s_2s_1$.
\end{enumerate}

Note that it is enough to consider the case when the number of tensor factors is 2. We first compute a finite set of numbers at which $V\Big(\pi_{a_1,1}^{(b_1)}\Big)\otimes V\Big(\pi_{a_2,1}^{(b_2)}\Big)=V_{a_1}(\omega_{b_1})\otimes V_{a_2}(\omega_{b_2})$ may fail to be cyclic. Set $\{a,b..c\}=\{a, b, b+1 \ldots, c-1, c\}$ and $\{a..\hat{b}..c\}=\{a..b-1\}\bigcup \{b+1..c\}$. We omit the detailed computations for the following lemma because the computations is rudimental. In Cases 5, 6 and 7, $S(b_1,b_2)=S(b_2,b_1)$, so we omit the redundant sets.

\begin{lemma}\label{mtoyspl}
Let $L=V_{a_1}(\omega_{b_1})\otimes V_{a_2}(\omega_{b_2})$ be an ordered tensor product of fundamental representations of $\yg$. If $a_2-a_1\notin S\left(b_1,b_2\right)$, then $L$ is a highest weight representation of $\yg$, where the set $S\left(b_1,b_2\right)$ is defined as follows:
\begin{enumerate}

  \item When $\g=\nyn$,
$S\left(b_1,b_2\right)=\left\{\frac{|b_2-b_1|}{2}+k|1\leq k\leq \text{min}\left\{b_1, l-b_2+1\right\}\right\}$.

\item When $\g=\nyso$,
  \begin{enumerate}
  \item $S\left(b_1,b_2\right)=\left\{|b_1-b_2|+2+2r, 2l-(b_1+b_2)+1+2r\right\}$, where $1\leq b_1, b_2\leq l-1$ and $0\leq r<\text{min}\{b_1, b_2\}$;
  \item  $S\left(l,b_2\right)=\left\{l-b_2+2+2r|0\leq r<b_2, 1\leq b_2\leq l-1\right\}$;
  \item $S\left(b_1,l\right)=\left\{l-b_1+1+r, l-b_1+r|0\leq r< b_1, 1\leq b_1\leq l-1\right\}$;
   \item $S\left(l,l\right)=\left\{1,3,\ldots,2l-1\right\}$.
\end{enumerate}

 \item When $\g=\nysp$, \begin{enumerate}
  \item $S\left(b_1,b_2\right)=\left\{\frac{|b_1-b_2|}{2}+1+r, l+2+r-\frac{b_1+b_2}{2}|0\leq r<\text{min}\{b_1, b_2\}\right\}$, where $1\leq b_1, b_2\leq l-1$;
  \item  $S\left(l,b_2\right)=\left\{\frac{l-b_2+1}{2}+1+r, \frac{l-b_2-1}{2}+1+r|0\leq r<b_2, 1\leq b_2\leq l-1\right\}$;
  \item $S\left(b_1,l\right)=\left\{\frac{l-b_1+1}{2}+2+r|0\leq r< b_1, 1\leq b_1\leq l-1\right\};$
   \item $S\left(l,l\right)=\left\{2,3,\ldots,l+1\right\}$.
\end{enumerate}

  \item When $\g=\nyo$,
  \begin{enumerate}
  \item $S\left(b_1,b_2\right)=\{\frac{|b_1-b_2|}{2}+1+r, l+r-\frac{b_1+b_2}{2}\}$, where $1\leq b_1, b_2\leq l-2$ and $0\leq r<\text{min}\{b_1, b_2\}$;
  \item  $S\left(l-1,b_2\right)=S\left(l,b_2\right)=S\left(b_2,l-1\right)=S\left(b_2,l\right)=\\ \left\{\frac{l-1-b_2}{2}+1+r|0\leq r<b_2, 1\leq b_2\leq l-2\right\}$;
  \item $S\left(l-1,l\right)=S\left(l,l-1\right)=\left\{2,4,\ldots,l-2+\overline{l}\right\}$,  where $\overline{l}=0$ if $l$ is even and let $\overline{l}=1$ if $l$ is odd;
   \item $S\left(l-1,l-1\right)=S\left(l,l\right)=\left\{1,3,\ldots,l-1-\overline{l}\right\}$.
\end{enumerate}
 \item When $\g=E_6$,

\[\begin{array}{lll}
S(1,1)=\{1,4\} &
S(1,2)=\{\frac{5}{2},\frac{9}{2}\} &
S(1,3)=\{\frac{3}{2},\frac{7}{2},\frac{9}{2}\}\\
S(1,4)=\{2,3,4,5\}&
S(1,5)=\{\frac{5}{2},\frac{7}{2},\frac{11}{2}\}&
S(1,6)=\{3,6\}\\
S(2,2)=\{1,3,4,6\}&
S(2,3)=\{2,3,4,5\}&
S(2,4)=\{\frac{3}{2},\frac{5}{2},\frac{7}{2},\frac{9}{2},\frac{11}{2}\}\\
S(2,5)=\{2,3,4,5\}&
S(2,6)=\{\frac{5}{2}, \frac{9}{2}\}&
S(3,3)=\{1, 2,3,4,5\}\\
S(3,4)=\{\frac{3}{2},\frac{5}{2},\frac{7}{2},\frac{9}{2},\frac{11}{2}\}&
S(3,5)=\{2,3,4,5,6\}&
S(3,6)=\{\frac{5}{2}, \frac{7}{2}, \frac{11}{2}\}\\
S(4,4)=\{1,2,3,4,5,6\}&
S(4,5)=\{\frac{3}{2},\frac{5}{2},\frac{7}{2},\frac{9}{2},\frac{11}{2}\}&
S(4,6)=\{2,3,4,5\}\\
S(5,5)=\{1,2,3,4,5\}&
S(5,6)=\{\frac{3}{2},\frac{7}{2},\frac{9}{2}\}&
S(6,6)=\{1,4\}.
\end{array}\]
\item When $\g=E_7$,
\[ \begin{array}{lll}
S(1,1)=\{1,4,6,9\},&
S(1,2)=\{\frac{5}{2},\frac{9}{2},\frac{11}{2}, \frac{15}{2}\},&
S(1,3)=\{\frac{3}{2},\frac{7}{2}..\frac{13}{2},\frac{17}{2}\},\\
S(1,4)=\{2..8\},&
S(1,5)=\{\frac{5}{2}.. \frac{15}{2}\},&
S(1,6)=\{3,4,6,7\},\\
S(1,7)=\{\frac{7}{2},\frac{13}{2}\},&
S(2,2)=\{1,3..7,9\},&
S(2,3)=\{2..8\},\\
S(2,4)=\{\frac{3}{2}..\frac{17}{2}\},&
S(2,5)=\{2..8\},&
S(2,6)=\{\frac{5}{2}..\frac{15}{2}\},\\
S(2,7)=\{3,5,7\},&
S(3,3)=\{1.. 9\},&
S(3,4)=\{\frac{3}{2}..\frac{17}{2}\},\\
S(3,5)=\{2..8\},&
S(3,6)=\{\frac{5}{2}.. \frac{15}{2}\},&
S(3,7)=\{3,4,6,7\},\\
S(4,4)=\{1..9\},&
S(4,5)=\{\frac{3}{2},.. \frac{17}{2}\},&
S(4,6)=\{2..8\},\\
S(4,7)=\{\frac{5}{2}..\frac{15}{2}\},&
S(5,5)=\{1..9\},&
S(5,6)=\{\frac{3}{2}..\frac{17}{2}\},\\
S(5,7)=\{2,4,5,6,8\},&
S(6,6)=\{1,2, 4,5,6,8,9\},&
S(6,7)=\{\frac{3}{2},\frac{9}{2},\frac{11}{2}, \frac{17}{2}\},\\
S(7,7)=\{1,5,9\}.
\end{array}\]

\item When $\g=E_8$,
\[\begin{array}{ll}
S(1,1)=\{1,4,6,7, 9,10,12,15\},&
S(1,2)=\{\frac{5}{2}..\hat{\frac{7}{2}}..\hat{\frac{25}{2}}.. \frac{27}{2}\},\\
S(1,3)=\{\frac{3}{2}..\hat{\frac{27}{2}}..\frac{29}{2}\},&
S(1,4)=\{2..14 \},\\
S(1,5)=\{\frac{5}{2}..\frac{27}{2}\},&
S(1,6)=\{3..13\},\\
S(1,7)=\{\frac{7}{2}..\hat{\frac{11}{2}}..\hat{\frac{21}{2}}..\frac{25}{2}\},&
S(1,8)=\{4,7,9,12\},\\
S(2,2)=\{1..\hat{2}..\hat{14}.. 15\},&
S(2,3)=\{2..14\},\\
S(2,4)=\{\frac{3}{2}..\frac{29}{2}\},&
S(2,5)=\{2..14\},\\
S(2,6)=\{\frac{5}{2}..\frac{27}{2}\},&
S(2,7)=\{3..13\},\\
S(2,8)=\{\frac{7}{2},\frac{11}{2}, \frac{15}{2},\frac{17}{2}, \frac{21}{2},\frac{25}{2}\},&
S(3,3)=\{1..15\},\\
S(3,4)=\{\frac{3}{2}.. \frac{29}{2}\},&
S(3,5)=\{2..14\},\\
S(3,6)=\{\frac{5}{2}..\frac{27}{2}\},&
S(3,7)=\{3..13\},\\
S(3,8)=\{\frac{7}{2}..\hat{\frac{11}{2}}..\hat{\frac{21}{2}}.. \frac{25}{2}\},&
S(4,4)=\{1..15\},\\
S(4,5)=\{\frac{3}{2}..\frac{29}{2}\},&
S(4,6)=\{2..14\},\\
S(4,7)=\{\frac{5}{2}..\frac{27}{2}\},&
S(4,8)=\{3..13\},\\
S(5,5)=\{1..15\},&
S(5,6)=\{\frac{3}{2}..\frac{29}{2}\},\\
S(5,7)=\{2..14\},&
S(5,8)=\{\frac{5}{2}..\hat{\frac{7}{2}}..\hat{\frac{25}{2}}..\frac{27}{2}\},\\
S(6,6)=\{1..15\},&
S(6,7)=\{\frac{3}{2},\hat{\frac{7}{2}}.. \hat{\frac{15}{2}}..\hat{\frac{25}{2}}..\frac{29}{2}\},\\
S(6,8)=\{2,5..\hat{8}..11, 14\},&
S(7,7)=\{1,2,5..\hat{8}..11,14, 15\},\\
S(7,8)=\{\frac{3}{2},\frac{11}{2}, \frac{13}{2}, \frac{19}{2},\frac{21}{2}, \frac{29}{2}\},&
S(8,8)=\{1,6,10,15\}.
\end{array}\]
\item When $\g=F_4$,
\[\begin{array}{lll}
S(1,1)=\{1,4,5,8\},&
S(1,2)=\{2..7\},&
S(1,3)=\{2..7\},\\
S(1,4)=\{\frac{5}{2},\frac{7}{2},\frac{11}{2},\frac{13}{2}\},&
S(2,1)=S(1,2),&
S(2,2)=\{1..8\},\\
S(2,3)=\{1..8\},&
S(2,4)=\{\frac{3}{2}..\frac{15}{2}\},&
S(3,1)=\{3,4,6,7\},\\
S(3,2)=\{2..\hat{4}..8\},&
S(3,3)=\{1..9\},&
S(3,4)=\{\frac{3}{2}..\hat{\frac{5}{2}}..\hat{\frac{15}{2}}..\frac{17}{2}\},\\
S(4,1)=\{\frac{7}{2},\frac{13}{2}\},&
S(4,2)=\{\frac{5}{2},\frac{9}{2},\frac{11}{2},\frac{15}{2}\},&
S(4,3)=S(3,4)\\
S(4,4)=\{1,4,6,9\}.&
\end{array}\]
\item When $\g=G_2$,
\[\begin{array}{lll}
S(1,1)=\{3,4,5,6\},&
S(1,2)=\{\frac{1}{2},\frac{3}{2},\frac{5}{2},\frac{7}{2},\frac{9}{2}\},\\

\quad S(2,1)=\{\frac{9}{2},\frac{13}{2}\},&
S(2,2)=\{1,3,4,6\}.
\end{array}\]
\end{enumerate}
\end{lemma}

Let $S=\{a_1,\ldots, a_n\}$. Define $S+r=\{a_1+r,\ldots, a_n+r\}$.
\begin{proposition}\label{tpotfihw}
An ordered tensor product $L=V\Big(\pi_{a_1,m_1}^{(b_1)}\Big)\otimes V\Big(\pi_{a_2,m_2}^{(b_2)}\Big)$ is a highest weight representation if $a_2-a_1\notin \bigcup\limits_{s=0}^{m_2-1}\Big(\bigcup\limits_{r=0}^{m_1-1} \big(S(b_1, b_2)-s+r\big)\Big)$, where $S(b_1, b_2)$ is defined as in Lemma \ref{mtoyspl}.
\end{proposition}
\begin{proof}
It follows from Proposition \ref{braidp21} that $T_{w}\big(\pi_{a_1,m_1}^{(b_1)}\big)=\prod\limits_{r=0}^{m_1-1} T_{w}\big(\pi_{a_1+r,1}^{(b_1)}\big)$. It is easy to see that $\widetilde{\Big(T_{w}\big(\pi_{a_1,m_1}^{(b_1)}\big)\Big)_{b_2}}$ is in general position with respect to $\widetilde{\Big(\pi_{a_2,m_2}^{(b_2)}\Big)_{b_2}}$ if and only if for each $0\leq r\leq m_1-1$ and $0\leq s\leq m_2-1$, $\widetilde{\Big(T_{w}\big(\pi_{a_1+r,1}^{(b_1)}\big)\Big)_{b_2}}$ is in general position with respect to $\widetilde{\Big(\pi_{a_2+s,1}^{(b_2)}\Big)_{b_2}}$. By the Lemma \ref{mtoyspl}, $(a_2+s)-(a_1+r)\notin S(b_1,b_2)$, which is equivalent to $a_2-a_1\notin \bigcup\limits_{s=0}^{m_2-1}\Big(\bigcup\limits_{r=0}^{m_1-1} \big(S(b_1, b_2)-s+r\big)\Big)$.
\end{proof}

Note that the dimension of the local Weyl module $W(\omega_i)$ for the current algebra $\g[t]$
and that of local Weyl module $W(\pi_{i,a})$ for $\yg$  are the same. It follows from Proposition \ref{tpotfihw} and Theorem 3.7 in \cite{TaGu} that
\begin{theorem}
Let $\g$ be a complex simple Lie algebra. Let $\pi=\big(\pi_1(u),\ldots, \pi_{l}(u)\big)$, where $\pi_i\left(u\right)=\prod\limits_{j=1}^{m_i}\left(u-a_{i,j}\right)$. Let $S$ be the multiset of the roots of these polynomials. Let $a_1=a_{i,j}$ be one of the numbers in $S$ with the maximal real part and let $b_1=i$. Inductively, let $a_r=a_{s,t}\left(r\geq 2\right)$ be one of the numbers in $S\setminus\{a_1, \ldots, a_{r-1}\}\ (r\geq 2)$ with the maximal real part and $b_r=s$. Let $k=m_1+\ldots+m_l$. Then the ordered tensor product $L=V_{a_1}(\omega_{b_1})\otimes V_{a_2}(\omega_{b_2})\otimes\ldots\otimes V_{a_k}(\omega_{b_k})$ is a highest weight representation of $\yg$, and its associated polynomial is $\pi$. Moreover, the local Weyl module $W(\pi)$ is isomorphic to $L$.
\end{theorem}

\end{document}